\documentclass{amsart}

\usepackage{latexsym}
\usepackage{amsfonts,amssymb} 
\usepackage{graphicx}
\usepackage{tikz}
\usepackage{caption}
\usepackage{subcaption}

\newcommand{\field}[1]{\mathbb{#1}}
\newcommand{\A}{\field{A}}
\newcommand{\C}{\field{C}}

\newcommand{\G}{\field{G}}

\newcommand{\N}{\field{N}}

\newcommand{\R}{\field{R}}

\newcommand{\Z}{\field{Z}}

\theoremstyle{plain}

\newtheorem{theorem}{Theorem}[section]
\newtheorem{proposition}[theorem]{Proposition}
\newtheorem{lemma}[theorem]{Lemma}
\newtheorem{corollary}[theorem]{Corollary}
\newtheorem{definition}[theorem]{Definition}
\newtheorem{remark}[theorem]{Remark}

\theoremstyle{definition}

\theoremstyle{remark}

\begin{document}

\makeatletter	   
\makeatother     

\title{Factorial Rational Varieties which Admit or Fail to Admit an Elliptic $\G_m$-Action}
\author{Gene Freudenburg \and Takanori Nagamine}
\date{\today} 
\thanks{The work of the second author was supported by a Grant-in-Aid for JSPS Fellows  (No. 18J10420) from the Japan Society for the Promotion of Science}

\pagestyle{plain}  

\begin{abstract} Over a field $k$, we study rational UFDs of finite transcendence degree $n$ over $k$. We classify such UFDs $B$ when $n=2$, $k$ is algebraically closed, and $B$ admits a positive $\Z$-grading, showing in particular that $B$ is affine over $k$. We also consider the Russell cubic threefold over $\C$, and the Asanuma threefolds over a field of positive characterstic, showing that these threefolds admit no elliptic $\G_m$-action. Finally, we show that, if $X$ is an affine $k$-variety and 
\[
X\times\A^m_k\cong_k\A^{n+m}_k
\]
then $X\cong_k\A^n_k$ if and only if $X$ admits an elliptic $\G_m$-action. 
\end{abstract}
\maketitle

\section{Introduction}
In his 1977 paper \cite{Mori.77}, Mori gives a classification of unique factorization domains (UFDs) which are finitely generated over a field $k$ and which admit a positive $\Z$-grading over $k$. Geometrically, these correspond to factorial affine $\G_m$-varieties with elliptic (or good) $\G_m$-actions; see {\it Section\,\ref{degree+grading}}. To each such ring $B$, Mori associates a unique natural number $m$ and a subring $B^{(m)}$ derived from the grading such that $B=B^{(m)}[{\bf v}^{1/{\bf e}}]$, where {\bf v}
and {\bf e} are sequences encoding the ramification data for $B$. The subring $B^{(m)}$ is a UFD defined by a semicomplete polarized $k$-variety $(X,L)$; see {\it Remark\,\ref{polarized}}. 
The algebras $B$ are thus classified by certain semicomplete polarized $k$-varieties $(X,L)$ together with ramification data over the corresponding ring $R(X,L)$. Mori gives an explicit description of all such rings in the case $\dim_kB=2$ and $k$ is algebraically closed. 

Let $\mathcal{U}_k(n)$ denote the set of $k$-isomorphism classes of UFDs containing $k$ and of transcendence degree $n$ over $k$. 
Define the following subsets of $\mathcal{U}_k(n)$, where $B$ indicates a ring represented in the set.
\begin{enumerate}
\item $\mathcal{U}_k(n,{\bf A})$ : $B$ is affine over $k$
\item $\mathcal{U}_k(n,{\bf D})$ : $B$ admits a positive degree function over $k$
\item $\mathcal{U}_k(n,{\bf G})$ : $B$ admits a positive $\Z$-grading over $k$
\item $\mathcal{U}_k(n,{\bf P})$ : $B\subset k^{[m]}$ for some integer $m\ge n$
\item $\mathcal{U}_k(n,{\bf R})$ : $B$ is rational over $k$
\end{enumerate}
Here, $k^{[m]}$ denotes a polynomial ring in $m$ variables over $k$. Of course, there are other categories of interest, such as noetherian, 
regular or unirational UFDs, but the foregoing list is of primary interest for this paper. 

By a result of Eakin (\cite{Eakin.72}, Lemma B), definition (4) is equivalent to:
\begin{center} 
$(4)^{\prime}$  $\mathcal{U}_k(n,{\bf P})$ : $B\subset k^{[n]}$ 
\end{center}
Note the containments $\mathcal{U}_k(n,{\bf G})\subset\mathcal{U}_k(n,{\bf D})$ and $\mathcal{U}_k(n,{\bf P})\subset\mathcal{U}_k(n,{\bf D})$. If $[B]$ denotes the isomorphism class of the ring $B$, then the mapping $[B]\to \big[B^{[1]}\big]$ gives an inclusion $\mathcal{U}_k(n,{\bf *})\subset\mathcal{U}_k(n+1,{\bf *})$ for each of these five properties $({\bf *})$.  
We use the notation $\mathcal{U}_k(n,{\bf A},{\bf R})$ to denote $\mathcal{U}_k(n,{\bf A})\cap\mathcal{U}_k(n,{\bf R})$, etc. 
In this notation, Mori's paper describes $\mathcal{U}_k(n,{\bf A},{\bf G})$. 

For $n=1$ and $k$ algebraically closed, it is known that
\[
\mathcal{U}_k(1,{\bf A})=\Big\{ \big[ k[t]_{f(t)}\big] \,\big\vert\, f(t)\in k[t]\setminus\{ 0\}\Big\} 
\]
where $k[t]\cong k^{[1]}$ and $k[t]_{f(t)}$ denotes localization, and:
\[
 \mathcal{U}_k(1,{\bf D})=\mathcal{U}_k(1,{\bf G})=\mathcal{U}_k(1,{\bf P})=\Big\{ \big[ k^{[1]}\big] \Big\}
\]
See \cite{Freudenburg.17}, Lemma 2.9 and Lemma 2.12, and {\it Corollary\,\ref{dim-1}} below. 

In {\it Section\,\ref{Dim-2}}, we consider the family $\mathcal{B}(k)$ of two-dimensional affine $k$-domains $B$ defined as follows.
\begin{quote}
There exist $n\in\N$, 
pairwise relatively prime integers $a> b> c_1>\hdots > c_n\ge 2$, and distinct $1=\lambda_1,\hdots ,\lambda_n\in k^*$ 
such that:
\begin{equation}\label{type}
B=k[x,y,z_1,\hdots ,z_n]/(x^a+\lambda_iy^b+z_i^{c_i})_{0\le i\le n}
\end{equation}
\end{quote}
If $n=0$, then $B=k[x,y]\cong k^{[2]}$. If $n=1$, these are known as factorial {\bf Pham-Brieskorn surfaces}. The sequence $a,b,c_1,\hdots ,c_n$ is the sequence of {\bf ramification indices} for $B$. 

In \cite{Mori.77}, Theorem 5.1, Mori 
shows that $\mathcal{B}(k)\subset\mathcal{U}_k(2,{\bf A},{\bf G})$, with equality in the case $k$ is algebraically closed.\footnote{More precisely, 
the function $\mathcal{B}(k)\to \mathcal{U}_k(2,{\bf A},{\bf G})$ mapping $B$ to $[B]$ is injective, and when $k$ is algebraically closed, it is also surjective.} Mori's theorem thus shows $\mathcal{U}_k(2,{\bf A},{\bf G})\subset\mathcal{U}_k(2,{\bf R})$ when $k$ is algebraically closed. 

One of our main results is {\it Theorem\,\ref{classify}} below, which gives a complete description of $\mathcal{U}_k(2,{\bf G},{\bf R})$ in the case $k$ is algebraically closed, in particular, showing that
$\mathcal{B}(k)=\mathcal{U}_k(2,{\bf G},{\bf R})$. Consequently, $\mathcal{U}_k(2,{\bf G},{\bf R})\subset\mathcal{U}_k(2,{\bf A})$ in this case. 
Combining this with Mori's result, we conclude that, when $k$ is algebraically closed:
\[
\mathcal{U}_k(2,{\bf A},{\bf G})=\mathcal{U}_k(2,{\bf G},{\bf R})=\mathcal{U}_k(2,{\bf A},{\bf G},{\bf R})=\mathcal{B}(k)
\]
We would like to understand the larger set $\mathcal{U}_k(2,{\bf D},{\bf R})$, starting with the subset 
$\mathcal{U}_k(2,{\bf P})$. Note that, if $B\in\mathcal{U}_k(2,{\bf P})$, then $B$ is affine by Zariski's Theorem \cite{Zariski.54}, and 
$B$ is rational by Castelnuovo's Theorem \cite{Castelnuovo.1894}. Therefore, $\mathcal{U}_k(2,{\bf P})=\mathcal{U}_k(2,{\bf A},{\bf P},{\bf R})$.

One motivation to study $\mathcal{U}_k(n,{\bf P})$ is the fact that, if $A$ is the ring of invariants for a $\G_a$-action on the affine space $\A^{n+1}_k$, then $A\in\mathcal{U}_k(n,{\bf P})$, whereas $A\not\in\mathcal{U}_k(n,{\bf A})$ in general. 
For $n=2$, it known that $A\cong k^{[2]}$ when the characteristic of $k$ is zero (Miyanishi's Theorem \cite{Miyanishi.85}), but it is an open question whether this generalizes to all fields. 
For $n=3$, if $A$ is the ring of invariants for a $\G_a$-action on $\A^4_k$ and the characteristic of $k$ is zero, then 
$A\in\mathcal{U}_k(3,{\bf P},{\bf R})$ (rationality is due to Deveney and Finston \cite{Deveney.Finston.94a}), but it is not known if $A\in\mathcal{U}_k(3,{\bf A})$.
If the $\G_a$-action is homogeneous for a positive $\Z$-grading, then $A\in\mathcal{U}_k(3,{\bf G},{\bf P},{\bf R})$, but even here we do not know if $A$ is affine. 

The main tool in our proof of {\it Theorem\,\ref{classify}} is the theory of signature sequences, which is developed in {\it Section\,\ref{signature}}.
Signature sequences are defined for any pair $(B,\deg )$, where $B$ is an integral $k$-domain and $\deg$ is a non-negative degree function on $B$, but they have especially strong properties when $B$ is a UFD and $\deg$ is positive. 
{\it Section\,\ref{UFD}} introduces certain criteria for a ring to be a UFD, and {\it Section\,\ref{degree+grading}} discusses degree functions and gradings. 

When $X$ is a {\it smooth} affine variety over $\C$, then $X$ is a topological manifold, and the existence of an elliptic $\C^*$-action on $X$ is a strong form of contractibility: In this case, 
the $\C^*$-action has a unique (attractive) fixed point $x_0\in X$. If the action is given by $^{\lambda}x$ ($\lambda\in\C^*$, $x\in X$), then since all the weights of the action are positive integers, restriction to the real interval $t\in (0,1]$ yields:
\[
\lim_{t\to 0^+}(^tx) = x_0 \quad \forall x\in X
\]
So the requisite contracting homotopy is given by $F:X\times [0,1]\to X$, where
\[
F(x,t)=\begin{cases} ^tx\quad (t\ne 0) \cr x_0 \quad (t=0) \end{cases}
\]
A well-known theorem of Ramanujam \cite{Ramanujam.71} says that a smooth affine surface over $\C$ which is contractible and simply connected at infinity is isomorphic to $\C^2$. This can be used to show that any smooth affine surface over $\C$ with an elliptic $\C^*$-action is isomorphic to $\C^2$; see \cite{Flenner.Zaidenberg.06}. 

In the same paper, Ramanujam showed that any smooth contractible affine variety over $\C$ of dimension $n\ge 3$ is diffeomorphic to $\R^{2n}$, and is therefore either isomorphic to $\C^n$ or an exotic structure on $\C^n$. 
A well-known example of this phenomenon is the Russell cubic threefold $X$, which is discussed in {\it Section\,\ref{Dim-3}}. 
For the coordinate ring $B$ of $X$, it is known that $B\in\mathcal{U}_{\C}(3,{\bf A},{\bf R})$, that $X$ is smooth and contractible, and that $X\not\cong\C^3$. 
So $X$ is an exotic structure on $\C^3$. In {\it Theorem\,\ref{Russell}},
we show that $X$ does not have the stronger form of contractibility imposed by an elliptic $\C^*$-action, i.e., $B\not\in\mathcal{U}_{\C}(3,{\bf G})$. 

Similarly, we consider the Asanuma threefolds over a field of positive characteristic, showing that these also do not admit an elliptic $\G_m$-action ({\it Corollary\,\ref{Asanuma}}). 
This result is a consequence of {\it Theorem\,\ref{cancel}}, which highlights the role of elliptic $\G_m$-actions:
\begin{quote}
For any field $k$ and positive integers $n,m$, let $X$ be an affine $k$-variety such that $X\times\A^m_k\cong_k\A^{n+m}_k$. Then $X\cong\A^n_k$ if and only if $X$ admits an elliptic $\G_m$-action.
\end{quote}
In one direction, the condition $X\times\C^m\cong\C^{n+m}$ ensures that $X$ is smooth, affine and contractible, but does not imply the stronger condition $\C [X]\in\mathcal{U}_{\C}(n,{\bf A},{\bf G})$. 
In the other direction, if $B\in\mathcal{U}_{\C}(n,{\bf A},{\bf G})$ and $X={\rm Spec}(B)$ is smooth, then either $X\cong\C^n$ or $X$ is an exotic structure on $\C^n$. 
In {\it Section\,\ref{conjecture}}, we conjecture the following characterization of affine space:
\begin{quote}
Let $k$ be an algebraically closed field, and let $X$ be a factorial rational affine $k$-variety of dimension $n$. If $X$ is smooth and admits an elliptic $\G_m$-action, then $X\cong_k\A^n_k$.
\end{quote}
The conjecture is true for $n=1$ and $n=2$. 
\medskip

\noindent {\bf Preliminaries.} For the integral domain $B$ and integer $n\ge 0$, $B^*$ is the group of units of $B$ and $B^{[n]}$ is the polynomial ring in $n$ variables over $B$. If $K$ is a field, then $K^{(n)}$ denotes the field of fractions of $K^{[n]}$. 
For a ground field $k$, affine space $n$-space over $k$ is denoted by $\A^n_k$, $\G_a$ is the additive group of $k$, and $\G_m$ the multiplicative group of $k^*$. 
If $B$ is a $k$-algebra, the {\bf Makar-Limanov invariant} $ML(B)$ of $B$ is the intersection of all invariant rings of $\G_a$-actions on $B$, 
and the {\bf Derksen invariant} $\mathcal{D}(B)$ of $B$ is the subring generated by invariants of non-trivial $\G_a$-actions. 
$B$ is {\bf rigid} if $ML(B)=B$, and {\bf stably rigid} if $ML(B^{[n]})=B$ for every $n\ge 0$. See \cite{Freudenburg.17} for details. 


\section{Criteria for a Ring to be a UFD}\label{UFD}

Let $A$ be an integral domain. It is well-known that, if $A$ is a UFD, then every localization of $A$ is a UFD. 
A partial converse is given by Nagata in \cite{Nagata.57}, Lemma 2. 
\begin{theorem} {\bf (Nagata's Criterion)} Let $A$ be a noetherian integral domain and $S\subset A\setminus\{ 0\}$ a multiplicatively closed set generated by a set of prime elements of $A$. Then $A$ is a UFD if and only if $S^{-1}A$ is a UFD.
\end{theorem}
The main purpose of this section is to introduce two additional criteria for a ring to be a UFD. 

\subsection{Integral Extensions}
The following result generalizes Samuel \cite{Samuel.64}, Theorem 8.1.
\begin{theorem}\label{Samuel} Let $A=\bigoplus_{i\in\Z}A_i$ be a $\Z$-graded integral domain which is finitely generated as an $A_0$-algebra, and let
$F\in A_{\omega} \setminus \{ 0 \}$, $\omega\in\Z$. Define $B=A[Z]/(Z^c-F)$, where $A[Z]=A^{[1]}$, $c\in\N$ and $\gcd(c,\omega )=1$. 
\begin{itemize}
\item [{\bf (a)}] $B$ is an integral domain and ${\rm frac}(A)\cong_K{\rm frac}(B)$, where $K={\rm frac}(A_0)$.
\item [{\bf (b)}] If $F$ is prime in $A$, then $z$ is prime in $B$, where $z=\pi (Z)$ for the surjection $\pi :A[Z]\to B$.
\item [{\bf (c)}] If $A$ is noetherian and $F$ is prime in $A$, then $A$ is a UFD if and only if $B$ is a UFD.
\end{itemize}
\end{theorem}

\begin{proof} 
Let $x_1,\hdots ,x_n\in A$ be such that $A=A_0[x_1,\hdots ,x_n]$, and let $F=P(x_1,\hdots ,x_n)$ for $P\in A_0^{[n]}$. 
Set $\omega_i=\deg x_i$, $1\le i\le n$. 

Consider first the case $c=d\omega\pm 1$ for some $d\in\N$. Define an $A_0$-automorphism $\varphi$ of $A[Z,Z^{-1}]$ by
\[
x_i^{\prime}=\varphi (x_i)=Z^{-d\omega_i}x_i\,\, (1\le i\le n) \quad {\rm and}\quad \varphi (Z)=Z
\]
and define $\tilde{F}= \varphi (F)=P(x_1^{\prime},\hdots ,x_n^{\prime})$. 
Set $A'=\varphi (A)=A_0[x_1^{\prime},\hdots ,x_n^{\prime}]$.
We have:
\begin{eqnarray*}
Z^c-F &=& Z^c-P(Z^{d\omega_1}x_1^{\prime},\hdots ,Z^{d\omega_n}x_n^{\prime}) \\
&=& Z^c-Z^{d\omega}P(x_1^{\prime},\hdots ,x_n^{\prime}) \\
&=& Z^{d\omega}(Z^{\pm 1}-\tilde{F}) 
\end{eqnarray*}
Therefore:
\[
B[z^{-1}]=A[Z,Z^{-1}]/(Z^c-F)=A'[Z,Z^{-1}]/(Z^{\pm 1}-\tilde{F})=A'[\tilde{F},\tilde{F}^{-1}]
\]
It follows that $B$ is an integral domain and ${\rm frac}(B)={\rm frac}(A')\cong_K{\rm frac}(A)$. So statement (a) holds in this case. 

In general, there exist $j,d\in\Z$ with $j\ge 0$ such that $jc=d\omega\pm 1$. Consider the ring 
\[
R:=A[T]/(T^{jc}-F)=A[t]
\]
where $A[T]=A^{[1]}$ and $t=p(T)$ for the canonical surjection $p:A[T]\to R$. 
By what was shown above, $R$ is an integral domain and ${\rm frac}(A)\cong_K{\rm frac}(R)$. 

For the subring $A[t^j]\subset R$ we have:
\[
A[t^j]=A[Z]/(Z^c-F)=B=A[z] \implies R=B[S]/(S^j-z)=B[t]
\]
where $B[S]=B^{[1]}$. Therefore, $B$ is an integral domain. Let $\psi$ be an $A$-automorphism of the localization $B[S,S^{-1}]$ defined by:
\[
\tilde{z}=\psi (z)=S^{1-j}z \quad {\rm and}\quad \psi (S)=S
\]
Set $B'=\psi (B)=A[\tilde{z}]$. Then $B'\cong_{A_0}B$ and $B[S,S^{-1}]=B'[S,S^{-1}]$. In addition:
\[
S^j-z=S^j-S^{j-1}\tilde{z}=S^{j-1}(S-\tilde{z})
\]
Therefore:
\[
R[t^{-1}]=B[t,t^{-1}]=B[S,S^{-1}]/(S^j-z)=B'[S,S^{-1}]/(S-\tilde{z})=B'[\tilde{z},\tilde{z}^{-1}]
\]
Consequently, ${\rm frac}(A)\cong_K{\rm frac}(R)={\rm frac}(B')\cong_K{\rm frac}(B)$.
This completes the proof for part (a). 

For part (b), assume that $F$ is prime in $A$. Since $B/zB\cong A/FA$, $z$ is prime in $B$.

For part (c), assume that $A$ is noetherian and $F$ is prime in $A$. Since $z$ is a prime element of $B$, $\tilde{z}$ is a prime element of $B'$. 

Consider first the case $c=d\omega\pm 1$ for some $d\in\N$.
If $A$ is a UFD, then $A'$ is a UFD, as is $A'[\tilde{F},\tilde{F}^{-1}]=B[z^{-1}]$. Since $z\in B$ is prime, it follows by Nagata's criterion that $B$ is a UFD.
Conversely, assume that $B$ is a UFD. Then the localization $B[z^{-1}]=A'[\tilde{F},\tilde{F}^{-1}]$ is a UFD. Since $\tilde{F}$ is prime in $A'$, it follows by Nagata's criterion that $A'$, hence $A$, is a UFD. 
So statement (c) holds in this case. 

In general, assume $j,d\in\Z$, $j\ge 0$, are such that $jc=d\omega\pm 1$. For the ring $R$ as above, we have shown that $t$ is prime in $R$, and that $R$ is a UFD if and only if $A$ is a UFD. 

If $B$ is a UFD, then $B'$ is a UFD, as is $B'[\tilde{z},\tilde{z}^{-1}]=R[t^{-1}]$. Since $t\in R$ is prime, it follows by Nagata's criterion that $R$ is a UFD.

Conversely, assume that $R$ is a UFD. Then the localization $R[t^{-1}]=B'[\tilde{z},\tilde{z}^{-1}]$ is a UFD. Since $\tilde{z}$ is prime in $B'$, it follows by Nagata's criterion that $B'$, hence $B$, is a UFD. 

We have thus shown: $A$ is a UFD if and only if $R$ is a UFD if and only if $B$ is a UFD. So statement (c) is true in the general case. 
\end{proof}

Note that, although ${\rm frac}(A)\cong_K{\rm frac}(B)$ in the theorem above, the inclusion $A\subset B$ is not birational if $c\ge 2$.

Let $\mathfrak{g}$ be the $\Z$-grading of $A$ in {\it Theorem\,\ref{Samuel}}.
Extend the $\Z$-grading $c\,\mathfrak{g}$ of $A$ to a $\Z$-grading $\mathfrak{g}^{\prime}$ of $A[Z]$ by letting $Z$ be homogeneous and:
\[
\deg_{\mathfrak{g}^{\prime}}Z=\deg_{\mathfrak{g}}F=\omega
\]
Then $Z^c-F$ is homogeneous and the quotient $B$ has the $\Z$-grading induced by $\mathfrak{g}^{\prime}$. 

\subsection{Affine Modifications of UFDs} If $A$ is an integral domain, $I\subset A$ is an ideal, and $f\in I$ is nonzero, then the {\bf affine modification} of $A$ along $f$ with center $I$ is 
the subring of the localization $A_f$ defined by:
\[
A[f^{-1}I]=A[a/f\,\vert\, a\in I]
\]
The reader is referred to \cite{Kaliman.Zaidenberg.99} for the theory of affine modifications.

The following result generalizes Nagata \cite{Nagata.57}, Theorem 1.

\begin{theorem}\label{affine-mod} Let $A$ be a noetherian UFD, $I\subset A$ an ideal, and $f\in I$. Assume that there exist $a_1,\cdots a_n\in A$ such that:
\begin{enumerate}
\item $I=(f,a_1,\hdots ,a_n)$
\item $\gcd(f,a_1\cdots a_n)=1$ 
\item $(p,a_1,\hdots ,a_n)$ is a prime ideal of $A$ for every prime divisor $p\in A$ of $f$
\end{enumerate}
Then $A[f^{-1}I]$ is a UFD. Moreover, any $\Z$-grading of $A$ for which $f,a_1,\hdots ,a_n$ are homogeneous extends to a $\Z$-grading of $A[f^{-1}I]$.
\end{theorem}

\begin{proof} Let $B=A[f^{-1}I]= A[a_1/f,\hdots ,a_n/f]$ and $A[Z_1,\hdots ,Z_n]=A^{[n]}$. Since $\gcd (f,a_i)=1$ for each $i$, the ring 
\[
A[Z_1,\hdots ,Z_n]/(fZ_1-a_1,\hdots ,fZ_n-a_n)
\]
is an integral domain isomorphic to $B$.
Let $p\in A$ be a prime divisor of $f$. Then:
\begin{eqnarray*}
B/pB&\cong&A[Z_1,\hdots ,Z_n]/(fZ_1-a_1,\hdots ,fZ_n-a_n,p) \\
&\cong& A/(a_1\hdots ,a_n,p)[Z_1,\hdots ,Z_n] \\
&\cong& A/(a_1,\hdots ,a_n,p)^{[n]}
\end{eqnarray*}
Since $a_1A+\cdots +a_nA+pA$ is a prime ideal of $A$, $pB$ is a prime ideal of $B$. 

Let $S\subset A$ be the multiplicatively closed set generated by the prime divisors of $f$. We have:
\[
B=A[a_1/f,\hdots ,a_n/f]\subset S^{-1}A=S^{-1}B
\]
Since $A$ is a UFD, $S^{-1}A=S^{-1}B$ is a UFD. By Nagata' criterion, $B$ is a UFD. 

Assume that $A$ has $\Z$-grading $\mathfrak{g}$. Extend $\mathfrak{g}$ to a $\Z$-grading $\mathfrak{g}^{\prime}$ of $A[Z_1,\hdots ,Z_n]$ by letting $Z_i$ be homogeneous 
with:
\[
\deg_{\mathfrak{g}^{\prime}}Z_i=\deg_{\mathfrak{g}}a_i-\deg_{\mathfrak{g}}f
\]
Then $(fZ_1-a_1,\hdots ,fZ_n-a_n)$ is a homogeneous ideal, and the quotient $A[f^{-1}I]$ has the $\Z$-grading induced by $\mathfrak{g}^{\prime}$. 
\end{proof} 

\subsection{An Application}

The following lemma generalizes Lemma 2 in \cite{Daigle.Freudenburg.01a}. 

\begin{lemma}\label{prime} Let $A_0$ be an integral domain. Given the integer $n\ge 0$, let 
\[
R_n=A_0[z_0,\hdots ,z_n]\cong A_0^{[n+1]}
\]
and let $a_1,\hdots ,a_n,b_1,\hdots ,b_n$ be positive integers such that 
$\gcd (a_i,b_1\cdots b_i)=1$ for each $i$. The ideal
\[
I_n=(z_1^{a_1}+z_0^{b_1},\hdots , z_n^{a_n}+z_{n-1}^{b_n})
\]
is a prime ideal of $R_n$. 
\end{lemma}

\begin{proof} We proceed by induction on $n$, the case $n=0$ being clear: $I_0=(0)$. 

Assume, for some $n\ge 1$, that $I_{n-1}$ is a prime ideal of $R_{n-1}$.
Define a $\Z$-grading of $R_{n-1}$ over $A_0$ for which $z_i$ is homogeneous of degree $b_1\cdots b_ia_{i+1}\cdots a_{n-1}$, $0\le i\le n-1$.
Then the quotient ring $A:=R_{n-1}/I_{n-1}$ is a $\Z$-graded integral domain which is finitely generated over $A_0$.

Let $F\in A$ be the image of $z_{n-1}^{b_n}$, noting that $\deg F=b_1\cdots b_n$. By hypothesis, $\gcd (a_n,\deg F)=1$. Therefore, by {\it Theorem\,\ref{Samuel}(a)}, the ring
$A[Z]/(Z^{a_n}-F)\cong R_n/I_n$ is an integral domain. 

It follows by induction that $I_n$  is a prime ideal of $R_n$ for each integer $n\ge 0$.
\end{proof}

\begin{theorem}\label{threefold} Let $K$ be a noetherian UFD. Given the integer $n\ge 0$, let $K[z_0,\hdots ,z_{n+1}]\cong K^{[n+2]}$, and let $a_1,\hdots ,a_n,b_1,\hdots ,b_n$ be positive integers such that $\gcd (a_i,b_1\cdots b_i)=1$ for each $i$. Given nonzero $f\in K$, the ring
\[
A_n:=K[z_0,\hdots ,z_{n+1}]/(fz_{i+1}+z_i^{a_i}+z_{i-1}^{b_i})_{0\le i\le n}
\]
is a UFD whose field of fractions equals  ${\rm frac}(K[z_0,z_1])\cong ({\rm frac}\, K)^{(2)}$. 
\end{theorem}

\begin{proof} If $f\in K^*$, then $A_n\cong K^{[2]}$ is a UFD. So assume that $f$ is not a unit of $K$. 

We proceed by induction on $n$, the case $n=0$ being clear. Note that each ring $A_m$ is noetherian, $0\le m\le n$. Given $m\ge 1$, assume that $A_{m-1}$ is a UFD. 
Let $p\in K$ be a prime divisor of $f$. Then
\[
A_{m-1}/(p,z_m^{a_m}+z_{m-1}^{b_m})\cong
(K/pK)[Z_0,\hdots ,Z_m]/(Z_1^{a_1}+Z_0^{b_1},\hdots ,Z_m^{a_m}+Z_{m-1}^{b_m})
\]
where $(K/pK)[Z_0,\hdots ,Z_m]\cong (K/pK)^{[m+1]}$. By {\it Lemma\,\ref{prime}}, $(p,z_m^{a_m}+z_{m-1}^{b_m})$ is a prime ideal of $A_{m-1}$.
Define the ideal $I\subset A_{m-1}$ by $I=(f,z_m^{a_m}+z_{m-1}^{b_m})$. Since $A_m\cong A_{m-1}[f^{-1}I]$, it follows by  {\it Theorem\,\ref{affine-mod}}  that $A_m$ is a UFD.

Therefore, by induction $A_n$ is a UFD.
Since affine modifications preserve quotient fields, we see that ${\rm frac}(A_n)={\rm frac}(A_0)={\rm frac}(K[z_0,z_1])\cong ({\rm frac}\, K)^{(2)}$.
\end{proof}

Rings of the type described in this theorem are considered in {\it Section\,\ref{Dim-3}}, where $K=k^{[1]}$ for a field $k$. 


\section{Degree Functions, $G$-Gradings and $\G_m$-Actions}\label{degree+grading}

An abelian group $G$ is {\bf totally ordered} if $G$ has a total order $\le$ which is translation invariant: 
\begin{center}
For all $x,y,z\in G$, $x+z\le y+z$ implies $x\le y$.
\end{center}

\subsection{Degree Functions} 

Assume that $G$ is a totally ordered abelian group, and that $B$ is an integral domain with degree function $\deg :B\to G\cup\{ -\infty\}$. We say the $\deg$ {\bf has values in $G$}. 
The induced filtration is
\[
B=\bigcup_{g\in G}\mathcal{F}_g
\]
where the sets $\mathcal{F}_g=\{ b\in B\, |\, \deg b\le g\}$ are the associated {\bf degree modules}. The associated {\bf degree submodules} are:
\[
\mathcal{V}_g=\{ f\in B\,\vert\, \deg f<g\}\subset\mathcal{F}_g
\]
Note that $\deg$ can be extended to $K={\rm frac}(B)$ by letting $\deg (f/g)=\deg f-\deg g$ for $f,g\in B$, $g\ne 0$. 
Note also that, if $B$ is a field, then $\deg$ is a degree function on $B$ if and only if $(-\deg)$ is a valuation of $B$. 
\begin{definition} {\rm $\deg$ is {\bf non-negative} if $\mathcal{V}_0=\{ 0\}$.}
\end{definition}
\begin{proposition}\label{degree} With the assumptions and notation above:
\begin{itemize}
\item [{\bf (a)}] $\mathcal{F}_0$ is a subring of $B$ which is integrally closed in $B$.
\item [{\bf (b)}] $\mathcal{F}_g$ is an ideal of $\mathcal{F}_0$ for each $g\le 0$. 
\item [{\bf (c)}] $\mathcal{F}_g$ is an $\mathcal{F}_0$-module for each $g\in G$, and $\mathcal{V}_g$ is a submodule. 
\item [{\bf (d)}] If $\deg$ is non-negative, then $\mathcal{F}_0$ is factorially closed in $B$ and $B^*\subset\mathcal{F}_0$. 
\item [{\bf (e)}] If $\deg$ is non-negative and $B$ is a UFD, then $\mathcal{F}_0$ is a UFD.
\item [{\bf (f)}] If $B$ is a normal ring, then $\mathcal{F}_0$ is a normal ring. 
\item [{\bf (g)}] If $B$ is a field, then $\mathcal{F}_0$ is a valuation ring of $B$ and ${\rm frac}(\mathcal{F}_0)=B$. 
\end{itemize} 
\end{proposition} 

\begin{proof} Extend $\deg$ to $K={\rm frac}(B)$ and let $V=\{ f\in K\,\vert\, \deg f\le 0\}$. Then $V$ is a valuation ring of $K$, and $\mathcal{F}_0=V\cap B$. This proves parts (a), (f) and (g). Proofs for statements (b)-(e) are left to the reader.
\end{proof}

\subsection{$k$-Algebras}

Suppose that $B$ is an integral $k$-domain for a ground field $k$. $\deg$ is a degree function {\bf over} $k$ if $\deg (k^*)=\{ 0\}$. Hereafter, any degree function on $B$ is assumed to be over $k$ when $k$ is the ground field. In this case, each degree module $\mathcal{F}_g$ is a $k$-vector space, and the associated degree submodule $\mathcal{V}_g$ is a subspace of $\mathcal{F}_g$. Let $\mathcal{W}_g$ be a complementary subspace, that is:
\[
\mathcal{F}_g=\mathcal{V}_g\oplus\mathcal{W}_g
\]
Then $\deg b=g$ for every nonzero $b\in\mathcal{W}_g$.
\begin{definition} {\rm Let $\deg$ be a degree function on $B$ with values in $G$. 
\begin{enumerate} 
\item $\deg$ is {\bf positive} if it is non-negative and $\mathcal{F}_0=k$.
\item $\deg$ is of {\bf finite type} if $\dim_k\mathcal{F}_g<\infty$ for each $g\in G$. 
\end{enumerate}
}
\end{definition}

Note that these properties are preserved under restriction: If $A\subset B$ is a $k$-subalgebra, then the degree function $\deg\vert_A$ on $A$ is non-negative (respectively, positive, of finite type) if $\deg$ is non-negative (respectively, positive, of finite type).

\begin{lemma}\label{finite-type} If $\deg$ is of finite type, then $\deg$ is non-negative. 
\end{lemma}

\begin{proof} Given $f\in\mathcal{F}_g$ for $g<0$, we have:
\[
fk[f]\subset\mathcal{F}_g \implies \dim_k fk[f]<\infty\implies \dim_k k[f]<\infty 
\]
We conclude that $k[f]$ is a field. If $f\ne 0$, then $f\in k[f]^*$. But then
\[
f^{-1}\in k[f]\subset\mathcal{F}_0 \quad {\rm and}\quad \deg f^{-1}>0
\]
which is a contradiction. Therefore, $f=0$.
\end{proof}

\subsection{$G$-Gradings} 

Let $B$ be an integral $k$-domain and $G$ an abelian group (not necessarily torsion free).  
Let $\mathfrak{g}$ be a $G$-grading of $B$ over $k$:
\[
B=\bigoplus_{g\in G}B_g\,\, ,\,\, k\subset B_0
\]
If $G$ is torsion free, then any choice of total order on $G$ gives a degree function $\deg_{\mathfrak{g}}$ on $B$. 
In this case, given $f\in B$, $\bar{f}$ will denote the highest-degree homogeneous summand of $f$.
\begin{definition} {\rm Under the above hypotheses:
\begin{enumerate} 
\item $\mathfrak{g}$ is {\bf non-negative} if $B_g=\{ 0\}$ for $g<0$.
\item $\mathfrak{g}$ is {\bf positive} if it is non-negative and $B_0=k$.
\item $\mathfrak{g}$ is of {\bf finite type} if $\dim_kB_g<\infty$ for each $g\in G$. 
\end{enumerate}
}
\end{definition}

These properties are preserved under restriction to graded subgalgebras: If $A\subset B$ is a graded $k$-subalgebra, then the induced grading 
\[
A=\bigoplus_{g\in G}A_g \,\, ,\,\, A_g=B\cap B_g
\]
of $A$ is non-negative (respectively, positive, of finite type) if $\mathfrak{g}$ is non-negative (respectively, positive, of finite type).

Note also that, if $G$ is totally ordered, and if $\mathfrak{g}$ is non-negative (respectively, positive), then $\deg_{\mathfrak{g}}$ is non-negative (respectively, positive). 
Thus, for any $k$-subalgebra $A$ of $B$, if $\mathfrak{g}$ is non-negative (respectively, positive), then $\deg_{\mathfrak{g}}$ restricts to a non-negative (respectively, positive) degree function on $A$.

However, it can happen that $\mathfrak{g}$ is of finite type, while $\deg_{\mathfrak{g}}$ is not. For example, if $B=k[x,x^{-1}]$, the ring of Laurent polynomials with the standard $\Z$-grading, 
then the grading is of finite type, but the associated degree function is not non-negative, and therefore not of finite type. 
However, if $\mathfrak{g}$ is non-negative and of finite type, then $\deg_{\mathfrak{g}}$ is of finite type. 
\begin{lemma}\label{sparse} If $\mathfrak{g}$ is positive, then $B^*=k^*$.
\end{lemma} 

\begin{proof} Assume $\mathfrak{g}$ is positive, and let $u\in B^*$. If $\deg_{\mathfrak{g}} u>0$, then $\deg_{\mathfrak{g}}(u^{-1})<0$, which is impossible, since $\deg b<0$ implies $b=0$. Therefore, $u\in\mathcal{F}_0=B_0=k$. 
\end{proof}

\subsection{$\G_m$-Actions}
Assume that $B$ is an affine $k$-domain and set $X={\rm Spec}(B)$. Let
\[
\rho (\mathfrak{g}) :\G_m\to {\rm Aut}_k(X)
\]
be the 
$\G_m$-action of $X$ induced by the nonzero $\Z$-grading $\mathfrak{g}$ of $B$. Recall the following definitions. 
\begin{enumerate}
\item $\rho (\mathfrak{g})$ is {\bf effective} if $\gcd \{ \deg_{\mathfrak{g}}f\,\vert\, f\in B\setminus\{ 0\}\}=1$.
\item $\rho (\mathfrak{g})$ is {\bf elliptic} if either $\mathfrak{g}$ or $(\mathfrak{-g})$ is positive.
\item $\rho (\mathfrak{g})$ is {\bf parabolic} if either $\mathfrak{g}$ or $(\mathfrak{-g})$ is non-negative, but not positive.
\item $\rho (\mathfrak{g})$ is {\bf hyperbolic} if it is neither elliptic nor parabolic.
\item $\rho (\mathfrak{g})$ is {\bf good} if it is both elliptic and effective. 
\end{enumerate}
Note that any $\G_m$-action on $X$ is of the form $\rho (c\mathfrak{g})$ for some $\Z$-grading $\mathfrak{g}$ where $\rho (\mathfrak{g})$ is effective. 
Two $\G_m$-actions $\sigma ,\tau$ of $X$ are {\bf equivalent} if there exists a $\Z$-grading $\mathfrak{g}$ of $B$ and nonzero $c,d\in\Z$ such that $\sigma =\rho (c\mathfrak{g})$ and 
$\tau =\rho (d\mathfrak{g})$. In particular, the equivalence class of any nontrivial $\G_m$-action contains exactly two effective members, being of the form $\rho (\pm\mathfrak{g})$.

The following result is needed in {\it Section\,\ref{Dim-2}}, 
and is due to Flenner and Zaidenberg; see \cite{Flenner.Zaidenberg.05b}, Theorem\,3.3. 

\begin{theorem}\label{FZ}  Let $X$ be a normal affine surface over $\C$. 
If $\C [X]$ is rigid and $X\not\cong \C^*\times \C^*$, then all $\C^*$-actions on $X$ are equivalent. 
\end{theorem}


\section{Signature Sequences for Non-Negative Degree Functions}\label{signature}
\subsection{Definition and Basic Properties}
Let $k$ be a field, $B$ an integral $k$-domain, $G$ a totally ordered abelian group, and $\deg :B\to G\cup\{ -\infty\}$ a non-negative degree function with filtration:
\[
B=\bigcup_{g\in G}\mathcal{F}_g
\]
\begin{definition} 
A {\bf signature sequence} $\vec{h}=\{ h_i\}_{i\in I}$ for $(B,\deg )$ is a sequence $h_i\in B$ indexed by an interval $0\in I\subset \N$ such that:
\begin{enumerate}
\item $h_0=1$
\item For each $n\in I$ with $n\ge 1$, $h_n\in \mathcal{F}_{d_n}\setminus k[h_1,\hdots ,h_{n-1}]$ where:
\[
d_n=\min\{ g\in G\, |\, \mathcal{F}_g\not\subset k[h_1,\hdots , h_{n-1}]\}
\]
\end{enumerate}
The {\bf length} of $\vec{h}$ is $\vert\vec{h}\vert$, and $\vec{h}$ is {\bf finite} or {\bf infinite} depending on $\vert\vec{h}\vert$. $\vec{h}$ is 
{\bf complete} if $B=k[\vec{h}]$.
\end{definition}

Note that the degree sequence $\{d_n\}\subset G$ has $d_n\le d_{n+1}$. In addition, for $n\le \vert\vec{h}\vert$, the subsequence $\{ h_0,\hdots ,h_n\}$ is a signature sequence. 

In case the degree function is of the form $\deg_{\mathfrak{g}}$ for some $G$-grading $\mathfrak{g}$ of $B$, we say that $\vec{h}$ is a {\bf homogeneous} signature sequence if each $h_n$ is homogeneous. 

By {\it Lemma\,\ref{finite-type}}, if a degree function $\deg$ on $B$ is of finite type, then it is non-negative. So signature sequences can be formed for any pair $(B,\deg )$ for which $\deg$ is of finite type. 

\begin{lemma} If $\deg$ is a degree function on $B$ of finite type, then $(B,\deg )$ admits a complete signature sequence. If $B$ is of finite type over $k$, then $(B,\deg )$ admits a complete signature sequence which is finite.
\end{lemma}

\begin{proof} There are two cases to consider.

Case 1: There exists a complete finite signature sequence for $(B,\deg )$. 

Case 2: There is no complete finite signature sequence for $(B,\deg )$. In this case, 
any finite signature sequence $\{ h_0,\hdots ,h_n\}$ can be extended, that is, 
\[
d:=\min\{ g\in G\,\vert\, \mathcal{F}_g\not\subset k[h_1,\hdots ,h_n]\}
\]
exists, and we can choose $h_{n+1}\in \mathcal{F}_d\setminus k[h_1,\hdots ,h_n]$. By induction, there exists an infinite signature sequence $\vec{h}$. 
Since $\dim_k\mathcal{F}_g<\infty$ for each $g$, it follows that, given $g\in G$:
\[
\mathcal{F}_g\subset k[h_1,\hdots ,h_n] \quad {\rm for} \quad n\gg 0
\]
Therefore, $B=k[\vec{h}]$ and $\vec{h}$ is complete.
\end{proof}

Let $\vec{h}$ be a signature sequence of length $L$ for the pair $(B,\deg )$, with degree sequence $d_i$. Define subgroups $H_i,H\subset G$ by:
\[
H_i=\langle d_1,\hdots ,d_i\rangle\,\, , 1\le i\le L\,\, , \,\, {\rm and}\,\, H=\deg (B\setminus\{ 0\})
\]
\begin{proposition}\label{intersect} Let $\vec{h}$ be a signature sequence for $(B,\deg )$ of length at least $n$. 
\begin{itemize}
\item [{\bf (a)}] If $b\in B$ and $\deg b<d_n$, then $b\in k[h_1,\hdots ,h_{n-1}]$.
\item [{\bf (b)}] Given $g\in H_{n-1}$, write $\mathcal{F}_g=\mathcal{V}_g\oplus\mathcal{W}_g$. If $g\le d_n$, then:
\[
\mathcal{W}_g\cap k[h_1,\hdots ,h_{n-1}]\ne \{ 0\}
\]
\end{itemize}
\end{proposition}

\begin{proof} Assume $f\in B\setminus k[h_1,\hdots ,h_{n-1}]$, and set $g =\deg f$. Then:
\[
f\in \mathcal{F}_g\setminus k[h_1,\hdots ,h_{n-1}]\implies g\ge d_n
\]
This proves part (a).

For part (b), since $g\in H_{n-1}$, there exist $c_1,\hdots ,c_{n-1}\in\N$ such that $\deg (h_1^{c_1}\cdots h_{n-1}^{c_{n-1}})=g$. Therefore, there exist $v\in\mathcal{V}_g$ and nonzero $w\in\mathcal{W}_g$ such that $h_1^{c_1}\cdots h_{n-1}^{c_{n-1}}=v+w$. Consequently:
\[
w-h_1^{c_1}\cdots h_{n-1}^{c_{n-1}}=v\in\mathcal{V}_g \implies \deg (w-h_1^{c_1}\cdots h_{n-1}^{c_{n-1}})<g\le d_n
\]
Part (a) implies $w-h_1^{c_1}\cdots h_{n-1}^{c_{n-1}}\in k[h_1,\hdots ,h_{n-1}]$, so $w\in k[h_1,\hdots ,h_{n-1}]$. This proves part (b).
\end{proof}

\begin{corollary}\label{irred} Let $\vec{h}$ be a signature sequence for $(B,\deg )$ of length at least $n$.
If $\deg$ is positive, then $h_n+b$ is irreducible in $B$ whenever $\deg b<\deg h_n$.
\end{corollary}

\begin{proof}
Assume that $h_n+b=uv$ for $u,v\in B$. If $\deg u<d_n$ and $\deg v<d_n$, then by {\it Proposition\,\ref{intersect}(a)} we have:
\[
u,v,b\in k[h_1,\hdots ,h_{n-1}] \implies uv=h_n+b\in k[h_1,\hdots ,h_{n-1}] \implies h_n\in k[h_1,\hdots ,h_{n-1}] 
\]
a contradiction. Therefore, either $\deg u\ge d_n$ or $\deg v\ge d_n$. Assume that $\deg u\ge d_n$. Then:
\[
d_n\le d_n+\deg v\le \deg u +\deg v=\deg (h_n+b)=d_n \implies \deg v=0 \implies v\in k^*
\]
Likewise, $u\in k^*$ if $\deg v\ge d_n$. 
\end{proof}

Note that, if $\vec{h}$ is a homogeneous signature sequence for a positive $G$-grading, this corollary implies that any $\beta\in B$ with $\bar{\beta}=h_n$ is irreducible and $\beta-h_n\in k[h_1,\hdots ,h_{n-1}]$, where $\bar{\beta}$ denotes the highest degree homogeneous summand of $\beta$. 


\subsection{Signature Sequences in UFDs}\label{UFD-sig}

In this section, assume that the field $k$ is algebraically closed. 

\begin{theorem}\label{fact-closed} Let $B$ be a UFD over $k$ with a positive degree function $\deg$. Assume that $\vec{h}$ is a signature sequence for $(B,\deg )$. Given $h_n\in\vec{h}$ and $b\in B$ with $\deg b<\deg h_n$, $k[h_n+b]$ is factorially closed in $B$. Consequently, $h_n+b$ is a prime element of $B$.
\end{theorem}

\begin{proof} We may assume $n\ge 1$. Suppose that $uv\in k[h_n+b]$ for $u,v\in B\setminus k$, and let $h_n+b-\lambda$ be a divisor of $uv$, where $\lambda\in k$. By {\it Corollary\,\ref{irred}}, 
$h_n+b-\lambda$ is irreducible in $B$, and therefore prime in $B$. It follows that every prime factor of $u$ (respectively, $v$) is of the form $h_n+b-\lambda$ for some $\lambda\in k$. Therefore, $u\in k[h_n+b]$ (respectively, $v\in k[h_n+b]$). 
\end{proof}

Note that this result means that every term $h_i$ of the signature sequence $\vec{h}$ is prime. 

\begin{corollary}\label{two-element} Let $B$ be a UFD over $k$ with a positive degree function $\deg$. Assume that $\vec{h}$ is a signature sequence for $(B,\deg )$. Given $h_m,h_n\in\vec{h}$ with $0<m<n$, $k[h_m,h_n]\cong k^{[2]}$. 
\end{corollary}

\begin{proof} By {\it Theorem\,\ref{fact-closed}}, $k[h_m]$ is factorially closed in $B$, hence algebraically closed in $B$. Since $h_n\not\in k[h_m]$, it follows that $h_n$ is transcendental over $k[h_m]$. Since $m>0$, $k[h_m]\cong k^{[1]}$. Therefore, $k[h_m,h_n]\cong k^{[2]}$.
\end{proof}

\begin{corollary}\label{dim-1} $\mathcal{U}_k(1,{\bf G})=\mathcal{U}_k(1,{\bf D})=\Big\{ k^{[1]}\Big\}$
\end{corollary}

\begin{proof} Let $B\in\mathcal{U}_k(1,{\bf D})$ and 
let $\deg$ be a positive degree function on $B$. Since ${\rm tr.deg}_kB=1$, there exists a signature sequence $\vec{h}$ for $(B,\deg )$ of length at least one. 
By {\it Theorem\,\ref{fact-closed}}, $k[h_1]$ is factorially closed in $B$, hence algebraically closed in $B$. Since ${\rm tr.deg}_kB=1$, this means $B= k[h_1]\cong k^{[1]}$.  We thus have:
\[
\Big\{ k^{[1]}\Big\}\subset \mathcal{U}_k(1,{\bf G})\subset \mathcal{U}_k(1,{\bf D})\subset\Big\{ k^{[1]}\Big\}
\]
\end{proof}


\section{Rational UFDs of Transcendence Degree Two}\label{Dim-2}

Let $\mathcal{B}(k)$ be the family of rings defined in (\ref{type}) above. The main goal of this section is to prove the following classification. 
\begin{theorem}\label{classify} $\mathcal{B}(k)\subset\mathcal{U}_k(2,{\bf G},{\bf R})$. If $k$ is algebraically closed, then $\mathcal{B}(k)=\mathcal{U}_k(2,{\bf G},{\bf R})$. 
\end{theorem}

\begin{proof} Assume that $k$ is algebraically closed and $B\in\mathcal{U}_k(2,{\bf G},{\bf R})$. Let $\mathfrak{g}$ be a positive $\Z$-grading of $B$, given by $B=\oplus_{i\in\N}B_i$. 
If $B\cong k^{[2]}$, then $B\in\mathcal{B}(k)$. So assume that $B\not\cong k^{[2]}$. 

Since ${\rm tr.deg}_kB=2$ and $B\not\cong k^{[2]}$, there exists a homogeneous signature sequence $\vec{h}$ of $(B,\deg_{\mathfrak{g}})$ of length at least three. 
Let $f=h_1$ and  $g=h_2$. 
Set $K={\rm frac}(B)\cong k^{(2)}$ and let $K_0\subset K$ be the subfield: 
\[
K_0=\{ u/v\in K \,\vert\, u,v\in B_i , v\ne 0, i\in\N\}
\]
By \cite{Freudenburg.17}, Proposition 1.1(c), $K_0$ is algebraically closed in $K$. Since $k\ne K_0$ and $K_0\ne K$, we conclude that ${\rm tr.deg}_kK_0=1$. 
Since ${\rm frac}(B)=k^{(2)}$, L\"uroth's Theorem\index{L\"uroth's Theorem} implies that $K_0=k(\zeta )=k^{(1)}$ for some $\zeta\in K_0$. Let $\zeta =u/v$ for $u,v\in B_t$ with $\gcd (u,v)=1$ for positive $t\in\Z$.

Let $a\ge b\in \N$ be such that $d_1=\deg f=bd$ and $d_2=\deg g=ad$ for $d=\gcd (d_1,d_2)$ and $\gcd (a,b)=1$.
Then there exist standard homogeneous $F,G\in k[X,Y]=k^{[2]}$ of the same degree $r$ such that:
\[
\gcd (F(X,Y),G(X,Y))=1 \quad {\rm and}\quad f^aG(u,v)=g^bF(u,v)
\]
Let $L_i,M_j\in k[X,Y]$, $1\le i,j\le r$, be linear forms such that $F=L_1\cdots L_r$ and $G=M_1\cdots M_r$. 
If $\lambda\in B$ is prime and $F(u,v),G(u,v)\in\lambda B$, then $L_i(u,v),M_j(u,v)\in\lambda B$ for some $i,j$. 
Since $L_i$ and $M_j$ are linearly independent, $u,v\in\lambda B$, which implies $\lambda\in k^*$. 
Therefore, $\gcd_B(F(u,v),G(u,v))=1$. It follows that $f$ is the only prime divisor of $L_1(u,v)$ and $g$ is the only prime divisor of $M_1(u,v)$. If 
$L_1(u,v)=f^{a_1}$ and $M_1(u,v)=g^{b_1}$ for $a_1\le a$ and $b_1\le b$, then $\deg f^{a_1}=\deg g^{b_1}$ implies $a_1=a$ and $b_1=b$. 
Since $(L_1,M_1)\in GL_2(k)$, it follows that $K_0=k(f^a/g^b)$.

Let $\xi\in B$ be homogeneous and irreducible, where $\xi\in B_l$ for positive $l\in\Z$. Assume $\xi B\ne fB$ and $\xi B\ne gB$. We have $\xi^{d_1}/f^l\in K_0$. Reasoning as above, we conclude that there exists a linear form $N\in k[X,Y]$ and positive $e\in\Z$ such that:
\begin{equation}\label{form} 
\xi^e=N(f^a,g^b) 
\end{equation}
Moreover, $\gcd (e,ab)=1$: Assume that $s_0\in\Z$ is a prime dividing $e$ and $ab$. Since $\gcd (a,b)=1$, 
either $a\in s_0\Z$ or $b\in s_0\Z$. Suppose that $e=s_0e_0$ and $a=s_0a_0$ for integers $e_0,a_0$. Furthermore, set $s_1=\gcd (e_0, a_0)$, $e_0 = s_1e_1$ and $a_0=s_1a_1$ for integers $e_1,a_1$. Then $e=se_1$ and $a=sa_1$, where $s = s_0s_1$. Since $\gcd (e_1,a_1)=1$, the equation above then yields $g^{b_1}=\lambda \xi^{e_1}+\mu f^{a_1}$ for $b_1 \leq b$  and $\lambda ,\mu\in k^*$. Since $\deg g^{b_1}=\deg f^{a_1}$, we must have $a_1=a$ and $b_1=b$. But then $s=1$, which is impossible. Therefore, no such prime $s_0$ exists. 

Suppose that $a=b$. Then $a=b=1$ and $d_1=d_2=d$. By equation (\ref{form}), $e\deg\xi=d$ for all homogeneous primes $\xi\in B$. If $e>1$, this implies $\deg\xi <d=d_1$, a contradiction. So $e=1$ and $\xi\in k[f,g]$. But then $B=k[f,g]$, also a contradiction. Therefore, $a>b$. 

According to equation (\ref{form}), there exist integers $c_n\ge 2$ such that $h_n^{c_n}\in\langle f^a,g^b\rangle$. 
By {\it Theorem\,\ref{fact-closed}}, $h_n$ is prime for each $n\ge 1$. Therefore, $\deg h_n$ divides $abd$ for all $n\ge 3$, which implies that 
the sequence of degrees $d_n$ is bounded. 

Suppose, for some pair $m\ne n$, that $c_m=tp$ and $c_n=tq$ for $t = \gcd (c_m,c_n)$ and nonzero $p,q\in\N$. Being powers of distinct primes, we see that $h_m^{c_m}$ and $h_n^{c_n}$ are $k$-linearly independent. 
Therefore, from equation (\ref{form}) it follows that 
\[
f^a,g^b\in\langle h_m^{c_m} ,h_n^{c_n}\rangle =\langle h_m^{tp},h_n^{tq}\rangle \implies f^{a'},g^{b'}\in\langle h_m^p ,h_n^q\rangle
\]
for some $a' \leq a$ and $b' \leq b$. But then:
\[
a'\deg f=b'\deg g \implies a'=a\,\, {\rm and}\,\, b'=b \implies t=1
\]
Therefore, $\gcd (c_m,c_n)=1$ for all pairs $m\ne n$. In particular, this means $d_{n+1}\ne d_n$ for all $n\ge 0$.
So $d_n$ is a strictly increasing sequence, and $c_n$, $n\ge 3$, is a strictly decreasing sequence. 
Since $d_n$ is also bounded, we conclude $d_n$ is finite. Consequently, $(B,\deg_{\mathfrak{g}})$ admits a finite complete homogeneous signature sequence $\vec{h}$, and if the length of $\vec{h}$ is $n$, then
$B=k[f,g,h_3,\hdots ,h_n]$. 

By re-scaling equations from (\ref{form}), we may assume that $f^a+\kappa_iy^b+h_i^{c_i}=0$, where $\kappa_i\in k^*$, $3\le i\le n$. Replacing $y$ with $\kappa_3^{1/b}y$, we may assume $\kappa_3=1$. By linear independence, we see that $\kappa_i\ne \kappa_j$ if $i\ne j$. 
We may thus write
\[
B\cong_kk[x,y,z_3,\hdots ,z_n]/(x^a+\kappa_iy^b+z_i^{c_i})_{3\le i\le n} \quad (\kappa_i\in k^*, \kappa_i\ne\kappa_j  \, {\rm if} \, i\ne j)
\]
where $a>b>c_3>\cdots >c_n\ge 2$ are pairwise relatively prime integers. 
Therefore, $B\in\mathcal{B}(k)$. 

Conversely, for any field $k$, suppose that $B\in\mathcal{B}(k)$ has the form (\ref{type}), and consider subrings
$R_i=k[x,y,z_3,\hdots ,z_i]$, $2\le i\le n$. 
For $3\le i\le n$, define $F_i=x^a+\lambda_iy^b$.

We see that $R_2=k[x,y]\cong k^{[2]}$ is a rational UFD with the $\Z$-grading for which $\deg x=b$ and $\deg x=a$, and that $x^a+\mu y^b$ is irreducible and homogeneous in $R_2$ for every $\mu\in k^*$. 

Given $i$ with $2\le i\le n-1$, suppose that $R_i$ is a rational UFD with positive $\Z$-grading $\mathfrak{g}$, and that $x^a+\mu y^b$ is irreducible and homogeneous in $R_i$
for every $\mu\in k^*\setminus\{ \lambda_3,\hdots ,\lambda_i\}$. Since 
\[
R_{i+1}=R_i[z_{i+1}]=R_i[Z]/(Z^{c_{i+1}}-F_{i+1})
\]
it follows by {\it Theorem\,\ref{Samuel}} that $R_{i+1}$ is a rational UFD. 

Let $\mu\in k^*\setminus\{ \lambda_3,\hdots ,\lambda_{i+1}\}$ and consider $G:=x^a+\mu y^b\in R_{i+1}$. We have:
\[
R_{i+1}/GR_{i+1}=R[Z]/(Z^{c_{i+1}}-F_{i+1},G)=(R_i/GR_i)[Z]/(Z^{c_{i+1}}-(\lambda_{i+1}-\mu)y^b)
\]
By the inductive hypothesis, $R_i/GR_i$ is an integral domain. By {\it Theorem\,\ref{Samuel}(a)}, $R_{i+1}/GR_{i+1}$ is an integral domain. Therefore, $G$ 
is irreducible and homogeneous in $R_{i+1}$. 

Finally, we may extend the $\Z$-grading $c_{i+1}\mathfrak{g}$ on $R_i$ to a positive $\Z$-grading on $R_{i+1}$ by letting $z_{i+1}$ be homogeneous of degree equal to 
$\deg_{\mathfrak{g}}F_i$. 

It follows by induction on $i$ that $B=R_n\in\mathcal{U}_k(2,{\bf G},{\bf R})$ and that 
$x^a+\mu y^b$ is irreducible and homogeneous in $B$ for every $\mu\in k^*\setminus\{ \lambda_1,\hdots ,\lambda_n\}$. 

This completes the proof.
\end{proof}

Note that the positive $\Z$-grading on $B\in\mathcal{B}(k)$ as defined in (\ref{type}) is given by 
\begin{equation}\label{Z-grading}
\deg (x,y,z_1,\hdots ,z_n)=(N/a, N/b, N/c_1,\hdots ,N/c_n)
\end{equation}
where $N=abc_1\cdots c_n$ and $x,y,z_1,\hdots ,z_n$ are homogeneous. 

\begin{corollary} $\mathcal{U}_k(2,{\bf G},{\bf R})=\mathcal{U}_k(2,{\bf A},{\bf G},{\bf R})$
\end{corollary}


\begin{theorem}\label{generators}
Given $B\in\mathcal{B}(k)$ as defined in {\rm (\ref{type})},
the minimum number of generators of $B$ over $k$ is $n + 2$.  
\end{theorem}
\begin{proof}
By hypothesis, we have
\[
B=\C [x,y,z_1,\hdots ,z_n]/(x^a+\lambda_iy^b+z_i^{c_i})_{0\le i\le n} 
\]
where $a>b>c_1>\cdots >c_n\ge 2$ are pariwise relatively prime integers and $1=\lambda_1,\hdots ,\lambda_n\in k^*$ are distinct. 
Let $d$ be the minimum number of generators of $B$ over $k$. Then clearly $d\le n+2$. 
Set $X = {\rm Spec} (B)\subset \A^{n+2}$.  For $1 \leq i \leq n$, let $f_i = x^a + \lambda_iy^b + {z_i}^{c_i}$. Let $J$ be the Jacobian matrix of $(f_1,\cdots ,f_n)$, namely:
	\begin{center}
	  $\displaystyle J = {\Bigg (}\frac{\partial f_i}{\partial x}, \frac{\partial f_i}{\partial y}, \frac{\partial f_i}{\partial z_j}{\Bigg )_{1 \leq i,\:j \leq n}}$
	\end{center} 
Then $J$ is of dimension $(n + 2) \times n$. For a closed point $p \in X$, we denote $J(p)$ by the Jacobian matrix at $p$, that is:
\begin{center}
	  $\displaystyle J(p) = {\Bigg (}\frac{\partial f_i}{\partial x}(p), \frac{\partial f_i}{\partial y}(p), \frac{\partial f_i}{\partial x_j}(p) {\Bigg )_{1 \leq i,\:j \leq n}}$
	\end{center} 
Let $\mathfrak{m}_p$ be a maximal ideal of $B$ corresponding to the origin $p = (0, \ldots, 0) \in X$. Since $a>b>c_i\ge 2$ for each $i$, we see that ${\rm rank} (J(p)) = 0$, and we have: 
	\begin{center}
	  $\dim_k (\mathfrak{m}_p / {\mathfrak{m}_p}^2) = (n + 2) + {\rm rank} (J(p)) = n + 2$
	\end{center}
Therefore, the dimension of the tangent space at the origin $p$ is $n + 2$, which implies $d \geq n + 2$. 
\end{proof}

\begin{theorem}\label{stable} Assume that the characteristic of $k$ equals 0. Given $B\in\mathcal{B}(k)$, let $R$ be a UFD such that $B\subset R$ and $B$ is factorially closed in $R$. 
If $B\not\cong k^{[2]}$ and $B\not\cong k[x,y,z]/(x^5+y^3+z^2)$, then $B\subset ML(R)$. In particular, $B$ is rigid (respectively, stably rigid) in these cases. 
\end{theorem}

\begin{proof} Assume that
\[
B=\C [x,y,z_1,\hdots ,z_n]/(x^a+\lambda_iy^b+z_i^{c_i})_{0\le i\le n} 
\]
where $a>b>c_1>\cdots >c_n\ge 2$ are pariwise relatively prime integers and $1=\lambda_1,\hdots ,\lambda_n\in k^*$ are distinct. 
We may assume $n\ge 1$, since otherwise $B=k[x,y]$.
From the proof of {\it Theorem\,\ref{classify}}, we see that $x,y,z_1,\hdots ,z_n$ are distinct prime elements of $B$, since $x=h_1,y=h_2,z_1=h_3,\hdots ,z_n=h_{n+2}$ is a complete signature sequence in $B$. Since $B$ is factorially closed in $R$, it follows that $x,y,z_1,\hdots ,z_n$ are distinct primes in $R$. 

If $a^{-1}+b^{-1}+c_i^{-1}> 1$, then it is easy to check that $(a,b,c_i)=(5,3,2)$. By the ABC Theorem (\cite{Freudenburg.17}, Thm.\,2.48), it follows that $k[x,y,z_i]\subset ML(R)$ 
whenever $(a,b,c_i)\ne (5,3,2)$. Since $B$ is algebraic over $k[x,y,z_i]$ and $ML(R)$ is algebraically closed in $R$, it follows that $B\subset ML(R)$ if $(a,b,c_i)\ne (5,3,2)$ for all $i$.

If $(a,b,c_i)=(5,3,2)$ for some $i$, then $n=1$ and $(a,b,c_1)=(5,3,2)$.
\end{proof}

\begin{theorem}\label{non-iso} Elements of $\mathcal{B}(\C )$ are pairwise non-isomorphic as $\C$-algebras. 
\end{theorem}

\begin{proof} Let $B,B'\in\mathcal{B}(\C )$ be defined as in (\ref{type}):
\[
B=\C [x,y,z_1,\hdots ,z_n]/(x^a+\lambda_iy^b+z_i^{c_i})_{0\le i\le n} 
\]
where $a>b>c_1>\cdots >c_n\ge 2$ are pariwise relatively prime integers and $1=\lambda_1,\hdots ,\lambda_n\in\C^*$ are distinct; and 
\[
B'=\C [x',y',z_1^{\prime},\hdots ,z_m^{\prime}]/(x'^{a'}+\lambda_i^{\prime}y'^{b'}+(z_i^{\prime})^{c_i^{\prime}})_{0\le i\le m}
\]
where $a'>b'>c_1^{\prime}>\cdots >c_m^{\prime}\ge 2$ are pairwise relatively prime integers and $1=\lambda_1^{\prime},\hdots ,\lambda_m^{\prime}\in\C^*$ are distinct. 
By {\it Theorem\,\ref{generators}}, we must have $m=n$. If $m=n=2$, then $B,B'\cong \C^{[2]}$. So assume that $m=n\ge 3$.  
By {\it Theorem\,\ref{stable}}, the rings $B$ and $B'$ are rigid; see also {\it Remark\,\ref{Klein}} below. 

Assume that $\varphi :B'\to B$ is a $\C$-algebra isomorphism. Let $\mathfrak{g}$ and $\mathfrak{g}^{\prime}$ be the positive $\Z$-gradings of $B$ and $B'$, respectively, as given in (\ref{Z-grading}). In addition, let $\varphi (\mathfrak{g}')$ be the $\Z$-grading of $B$ induced by $\varphi$.
According to {\it Theorem\,\ref{FZ}}, there exists $\xi\in\Z$ such that $\varphi (\mathfrak{g}')=\xi\,\mathfrak{g}$. Since $\varphi$ is surjective, we see that 
$\xi=\pm 1$. If $\xi=-1$, we may 
compose $\varphi$ with an involution $\alpha$ of $B$ so that $\alpha\varphi (\mathfrak{g}')=\mathfrak{g}$. So we may assume that $\xi =1$, and $\varphi (\mathfrak{g}')=\mathfrak{g}$. 

From the proof of {\it Theorem\,\ref{classify}}, we have that $\vec{z}=\{ x,y,z_1,\hdots ,z_n\}$ is a homogeneous signature sequence for $(B,\mathfrak{g})$, and that 
$\vec{z'}=\{ x',y',z_1^{\prime},\hdots ,z_n^{\prime}\}$ is a homogeneous signature sequence for $(B',\mathfrak{g}^{\prime})=\varphi (B,\mathfrak{g})$. 
Therefore,
$\vec{h}:=\varphi \left(\vec{z'}\right)$ is also a homogeneous signature sequence for $(B,\mathfrak{g})$.  Write $\vec{h}=\{ f,g,h_3,\hdots , h_{n+2}\}$, where:
\[
\varphi (x')=f \,\, ,\,\, \varphi (y')=g\,\, ,\,\, \varphi (z_i^{\prime})=h_{i+2}\quad (1\le i\le n)
\]
Let $d_1=\deg x$, $d_2=\deg y$ and $d_{i+2}=\deg z_i$ for $i\ge 1$. 
The proof of {\it Theorem\,\ref{classify}} shows that the sequence $d_i$ is strictly increasing. Therefore, 
for each $i$ with $1\le i\le n$:
\[
\dim_{\C}B_{d_i}-\dim_{\C}(B_{d_i}\cap \C [x,y,z_1,\hdots ,z_{i-1}])=1
\]
If follows that:
\begin{enumerate}
\item There exist $u,v,w_i\in \C^*$ such that $f=ux$, $g=vy$ and $h_{i+2}=w_iz_i$ for $1\le i\le n$
\item $a=a'$, $b=b'$ and $c_i=c_i^{\prime}$ for $1\le i\le n$
\end{enumerate}
From the defining equations for $B$ and $B'$ we thus obtain for $1\le i\le n$:
\[
-z_i^{c_i}=x^a+\lambda_iy^b \quad {\rm and}\quad -h_{i+2}^{c_i}=f^a+\lambda_i^{\prime}g^b
\]
Therefore:
\[
	 - z_i^{c_i} = \frac{u^{a}}{w_i^{c_i}} x^{a} + \frac{v^{b}}{w_i^{c_i}} \lambda_i'y^{b} = x^a+\lambda_iy^b
\]
Since $x^a$ and $y^b$ are linearly independent over $k$, it follows that $u^a = w_i^{c_i}$ and $w_i^{c_i}\lambda_i = v^b\lambda_i'$ for $1 \leq i \leq n$. Set $\zeta = v^b/u^a \in \C^*$. 
Then $\lambda_i/\lambda_i^{\prime} =\zeta$ for $1 \leq i \leq n$. Since $\lambda_1 = \lambda_1' = 1$, we see that $\zeta = 1$, hence $\lambda_i = \lambda_i'$ for $1 \leq i \leq n$. 

This completes the proof of the theorem.
\end{proof}


\section{Rational UFDs of transcendence degree three}\label{Dim-3}

\subsection{Certain Affine Modifications of $k^{[3]}$}
Let $k[x]=k^{[1]}$ for a field $k$, and let $f=p(x)\in k[x]\setminus\{ 0\}$. 
Define the affine $k$-algebra
\begin{equation}\label{kernel}
B_n=k[x][z_0,\hdots ,z_{n+1}]/(p(x)z_{i+1}+z_i^{a_i}+z_{i-1}^{b_i})_{0\le i\le n}
\end{equation}
where $a_1,\hdots ,a_n,b_1,\hdots ,b_n$ are positive integers such that $\gcd (a_i,b_1\cdots b_i)=1$ for each $i$. 
Using $K=k[x]$ in {\it Theorem\,\ref{threefold}}, it follows that $B_n\in\mathcal{U}_k(3,{\bf A},{\bf R})$.
\begin{proposition}
If $p(x)\not\in k$ and $a_i,b_i\ge 2$ for all $i$, then the minimum number of generators of $B_n$ over $k$ is $n + 3$.  
\end{proposition}

\begin{proof}
Let $d$ be the minimum number of generators of $B_n$ over $k$. Then clearly $d\le n+3$. 
Set $X = {\rm Spec} (B_n)\subset \A^{n+3}$. For $0 \leq i \leq n$, let $f_i = p(x)z_{i+1}+z_i^{a_i}+z_{i-1}^{b_i}$. Let $J$ be the Jacobian matrix of $(f_0,\ldots ,f_n)$, namely:
	\begin{center}
	  $\displaystyle J = {\Bigg (}\frac{\partial f_i}{\partial x}, \frac{\partial f_i}{\partial z_j}{\Bigg )_{0 \leq i \leq n,\: 0 \leq j \leq n + 1}}$
	\end{center} 
Then $J$ is of a matrix of size $(n + 3) \times (n + 1)$. For $0 \leq i \leq n$ and $0 \leq j \leq n + 1$, we have $\partial f_i/ \partial x =  p'(x)z_{i + 1}$ and: 
\[
\frac{\partial f_i}{\partial z_j} = 
	\begin{cases}
	  p(x) 		  & (j = i + 1) \\ 
	  a_iz_i^{a_i - 1}  & (j = i) \\ 
	  b_iz_{i - 1}^{b_i - 1} & (j = i + 1) \\ 
	  0 & ({\rm otherwise})
  	 \end{cases}
\]
For a maximal ideal $\mathfrak{m}$ of $B_n$, we denote $J(\mathfrak{m})$ by the Jacobian matrix at $\mathfrak{m}$, that is,
\begin{center}
	  $\displaystyle J(\mathfrak{m}) = {\Bigg (}\frac{\partial f_i}{\partial x} (\mathfrak{m}), \frac{\partial f_i}{\partial z_j} (\mathfrak{m}){\Bigg )_{0 \leq i \leq n,\: 0 \leq j \leq n + 1}}$
	\end{center} 
where for $g \in B_n$, $g(\mathfrak{m})$ means the image of $g$ in $B_n/\mathfrak{m}$. 

Take a prime divisor $q(x) \in k[x]$ of $p(x)$, which is possible since $p(x)\not\in k$. Let $\mathfrak{m}$ be the maximal ideal of $B_n$ generated by $q(x), z_0, \ldots, z_{n + 1}$. 
Since $a_i,b_i\ge 2$ for each $i$, we see that ${\rm rank} (J(\mathfrak{m})) = 0$, hence we have: 
	\begin{center}
	  $\dim_k (\mathfrak{m} / {\mathfrak{m}}^2) = (n + 3) + {\rm rank} (J(\mathfrak{m})) = n + 3$
	\end{center}
Therefore, the dimension of the tangent space at $\mathfrak{m}$ is $n + 3$, which implies $d \geq n + 3$. 
\end{proof}

The threefolds listed in (\ref{kernel}) are of interest, since some of them occur as the kernel a of locally nilpotent derivation of $k^{[4]}$ when the characteristic of $k$ is 0. For instance, Example 8.11 and Example 8.15 of \cite{Freudenburg.17} give kernels isomorphic to 
\begin{equation}\label{3-kernel}
B_1=k[x,z_0,z_1,z_2]/(x^2z_2+z_1^2+z_0^3) \,\, {\rm and}\,\, B_2=k[x,z_0,z_1,z_2,z_3]/(xz_2+z_1^2+z_0^3,xz_3+z_2^2+z_1^3)
\end{equation}
respectively. $B_1$ has two independent positive $\Z$-gradings $\mathfrak{g}_1$ and $\mathfrak{g}_2$, where $x,z_0,z_1,z_2$ are homogeneous with:
\[
\deg_{\mathfrak{g}_1}(x,z_0,z_1,z_2)=(1,2,3,4) \,\, {\rm and}\,\, \deg_{\mathfrak{g}_2}(x,z_0,z_1,z_2)=(2,2,3,2)
\]
$B_2$ has positive $\Z$-grading $\mathfrak{h}$, where $x,z_0,z_1,z_2,z_3$ are homogeneous with:
\[
\deg_{\mathfrak{h}}(x,z_0,z_1,z_2,z_3)=(3,4,6,9,15) 
\]
For $n\ge 3$, it is easy to show that $B_n$ admits no positive $\Z$-grading for which $x,z_0,\hdots ,z_n$ are homogeneous.

\subsection{The Russell Cubic Threefold}
The Russell cubic threefold over $k$ is $X={\rm Spec}(B)$, where:
\[
B=k[x,y,z,t]/(x+x^2y+z^2+t^3) \in \mathcal{U}_k(3,{\bf A},{\bf R})
\]
$X$ is smooth and admits the hyperbolic $\G_m$-action $\rho (\mathfrak{g})$ induced by the $\Z$-grading 
$\mathfrak{g}$ of $B$ for which $x,y,z,t$ are homogeneous and $\deg_{\mathfrak{g}} (x,y,z,t)=(6,-6,3,2)$. 

Assume that $k=\C$. 
Dubouloz, Moser-Jauslin and Poloni describe the automorphism group $G={\rm Aut}_{\C}(B)$ in \cite{Dubouloz.Moser-Jauslin.Poloni.10} as follows.

It is known that $ML(B)=\C [x]$ and $\mathcal{D}(B)=\C [x,z,t]$. Thus, any element of $G$ restricts to both $\C [x]$ and $\C [x,z,t]\cong\C^{[3]}$. 
Define the ideal $I\subset \C [x,z,t]$ by $I=(x^2,z^2+t^3+x)$, and define the group:
\[
K=\{ \alpha\in {\rm Aut}_{\C [x]}\C [x,z,t]\,\vert\, \alpha (I)=I\, ,\, \alpha\equiv 1\, ({\rm mod}\, (x))\}
\]
Let $\varphi :\C^*\to {\rm Aut}_{\C}(K)$ be the restriction of $\rho (\mathfrak{g})$ to $\C [x,z,t]$. Then
\begin{equation}\label{decomp}
G\cong K\rtimes_{\varphi}\C^*
\end{equation}
where the isomorphism is gotten by restricting elements of $G$ to $\C [x,z,t]$. As a consequence, every automorphism of $X$ fixes the point $0\in X$ defined by the maximal ideal $(x,y,z,t)$ of $B$. 

\begin{theorem}\label{Russell} $B\not\in\mathcal{U}_{\C}(3,{\bf G})$
\end{theorem}

\begin{proof} 
Let $\mathfrak{h}$ be a positive $\Z$-grading of $B$, and let $\psi :\C^*\to G$ be the elliptic $\C^*$-action on $B$ induced by $\mathfrak{h}$. 
Then $\psi$ restricts to $R:=\C [x,z,t]$ and is completely determined by its action on $R$. 

Note first that $\psi$ fixes $\C [x]$, and $x$ must therefore be $\mathfrak{h}$-homogeneous.  
Since $\mathfrak{h}$ is positive, {\it Lemma\,\ref{Daigle}} below implies that there exist $\mathfrak{h}$-homogeneous $Z,T\in R$ with $R=\C [x,Z,T]$. 
Define $\beta\in K$ by $\beta (x,z,t)=(x,Z,T)$, and for $\lambda\in\C^*$, suppose that $\psi (\lambda )(x,Z,T)=(\lambda^px ,\lambda^qZ ,\lambda^rT)$ for positive $p,q,r\in\Z$. 
Given $\lambda\in\C^*$, write 
\[
\psi (\lambda)=\kappa_{\lambda}\varphi (\mu_{\lambda}) \,\, ,\,\, \kappa_{\lambda}\in K\,\, ,\,\, \mu_{\lambda}\in\C^*
\]
according to the decomposition of $G$ given in (\ref{decomp}). 
By \cite{Dubouloz.Moser-Jauslin.Poloni.10}, Proposition 3.6, there exist $f,g,\alpha\in\C [z,t]$ such that :
\[
\kappa_{\lambda}(z)=z+x(\alpha(z^2+t^3))_t+x^2f \quad {\rm and}\quad 
\kappa_{\lambda}(t)=t-x(\alpha(z^2+t^3))_z+x^2g 
\]
In addition $Z=\beta (z)$ and $T=\beta (t)$ also have this form. Therefore:
\begin{eqnarray*}
\psi (\lambda )(x)&=&\kappa_{\lambda}\varphi (\lambda )(\mu_{\lambda})(x)=\kappa_{\lambda}(\mu_{\lambda}^6x)=\mu_{\lambda}^6\kappa_{\lambda}(x)=\mu_{\lambda}^6x\\
\psi (\lambda )(z)&=&\kappa_{\lambda}\varphi (\lambda )(\mu_{\lambda})(z)=\kappa_{\lambda}(\mu_{\lambda}^3z)=\mu_{\lambda}^3\kappa_{\lambda}(z)=\mu_{\lambda}^3(z+x(\alpha(z^2+t^3))_t+x^2f)\\
\psi (\lambda )(t)&=&\kappa_{\lambda}\varphi (\lambda )(\mu_{\lambda})(t)=\kappa_{\lambda}(\mu_{\lambda}^2t)=\mu_{\lambda}^2\kappa_{\lambda}(t)=\mu_{\lambda}^2(t-x(\alpha(z^2+t^3))_z+x^2g)
\end{eqnarray*}
Let $V$ be the tangent space to ${\rm Spec}(R)\cong\C^3$ at $0$, and let $\Psi$ denote the $\C^*$-action on $V$ induced by $\psi\vert_R$. 
Since $x(\alpha(z^2+t^3))_z,x(\alpha(z^2+t^3))_t\in (x,z,t)^2$, it follows that:
\[
\Phi (\lambda)(x,Z,T)=(\mu_{\lambda}^6x , \mu_{\lambda}^3z , \mu_{\lambda}^2t)=(\lambda^px ,\lambda^qz ,\lambda^rt) 
\]
Therefore, $\psi (\lambda)(x,Z,T)=(\lambda^{6c}x,\lambda^{3c}Z,\lambda^{2c}T)$ for some $c\in\Z$, which implies:
\[
\psi\vert_R=\beta\left( \rho (c\mathfrak{g})\vert_R\right)\beta^{-1}=\left(\beta\rho (c\mathfrak{g})\beta^{-1}\right)\vert_R
\implies \psi = \beta\rho (c\mathfrak{g})\beta^{-1}
\]
But this is impossible, since $\psi$ is elliptic. Therefore, $B$ admits no positive $\Z$-grading.
\end{proof}

The following result, which is due to Daigle, implies that any elliptic $\G_m$-action on $\A_k^n$ is linearizable; see \cite{Freudenburg.17}, Proposition 3.42. 
\begin{lemma}\label{Daigle} Let $k$ be a field and $R=k^{[r]}$ $(r\ge 1)$, and let $\mathfrak{g}$ be a positive $\Z$-grading of $R$. If $f_1,\hdots ,f_n\in R$ are $\mathfrak{g}$-homogeneous and 
$R=k[f_1,\hdots ,f_n]$, then there is a subset $\{ g_1,\hdots ,g_r\}$ of $\{ f_1,\hdots ,f_n\}$ such that $R=k[g_1,\hdots ,g_r]$.
\end{lemma}

\subsection{Asanuma Threefolds}

Let $k$ be a field of characteristic $p>0$. In \cite{Asanuma.87}, Asanuma introduced the rational threefolds 
\[
A_m=k[x,y,z,t]/(x^my+f(z,t))
\]
where $m\ge 1$, $f(z,t)\in k[z,t]$ and $k[z,t]/(f)\cong_kk^{[1]}$ but $k[z,t]\ne k[f]^{[1]}$. Segre \cite{Segre.56} gives such non-standard line embeddings in $\A^2$, for example, defined by 
\[
f(z,t)=z^{p^e}+t+t^{sp}
\]
where $s,e\in\Z^+$ and $p^e$ and $sp$ do not divide each other;
see also the Introduction to \cite{Ganong.11}. Asanuma showed that $A_m^{[1]}\cong_kk^{[4]}$ for each $m\ge 1$. From this, it follows that 
$A_m\in\mathcal{U}_k(3,{\bf A},{\bf P},{\bf R})$ and that each threefold ${\rm Spec}(A_m)$ is smooth. 

These rings are considered by N. Gupta in \cite{Gupta.14a,Gupta.14c}, showing that, when $m\ge 2$, $ML(A_m)=k[x]$ and $\mathcal{D}(A_m)=k[x,z,t]$. So 
$A_m\not\cong_kk^{[3]}$ when $m\ge 2$, thus providing counterexamples for the cancellation problem for affine spaces in positive characteristic. It is an open problem whether $A_1\cong_kk^{[3]}$. 

\begin{theorem}\label{cancel} Let $k$ be a field and $A$ an affine $k$-domain, $n=\dim_kA$. The following conditions are equivalent. 
\begin{enumerate}
\item $A\in\mathcal{U}_k(n,{\bf G})$ and $A^{[m]}\cong_kk^{[n+m]}$ for some $m\in\N$ 
\item $A\cong_kk^{[n]}$
\end{enumerate}
\end{theorem}

\begin{proof} The implication (2) implies (1) is clear. For the converse, assume that condition (1) holds, and 
let $\mathfrak{g}$ be a positive $\Z$-grading of $A$ over $k$. There exist an integer $s\ge n$ and $\mathfrak{g}$-homogeneous elements $a_1,\hdots ,a_s\in A$ such that $A=k[a_1,\hdots ,a_s]$. 

Let $B=A[x_1,\hdots ,x_m]=A^{[m]}$, and let $I\subset B$ be the ideal $I=x_1B+\cdots +x_mB$. Extend the $\Z$-grading $2\mathfrak{g}$ on $A$ to a $\Z$-grading $\mathfrak{g}'$ of $B$ by letting each $x_i$ be homogeneous of degree one, $1\le i\le m$, and let $B=\bigoplus_{i\in\N}B_i$ be the decomposition of $B$ for $\mathfrak{g}'$. Then $B_0=k$ and $B_1=kx_1\oplus\cdots\oplus kx_m$. 
 
By {\it Lemma\,\ref{Daigle}}, there exists a subset $\{ g_1,\hdots ,g_{n+m}\}$ of $\{ a_1,\hdots ,a_s,x_1,\hdots ,x_m\}$ such that $B=k[g_1,\hdots ,g_{n+m}]$. 
Since $\mathfrak{g}'$ is positive and $\deg_{\mathfrak{g}'}a_i\ge 2$ for each $i$, it follows that:
\[
\{ x_1,\hdots ,x_m\}\subset \{ g_1,\hdots ,g_{n+m}\}
\]
By re-indexing the set $\{ g_1,\hdots ,g_{n+m}\}$, we may assume that $g_i=x_i$, $1\le i\le m$. 
Therefore, $B=k[x_1,\hdots ,x_m,g_{1+m},\hdots ,g_{n+m}]$, which implies:
\[
A\cong_kB/I\cong_kk[g_{1+m},\hdots ,g_{n+m}]/I\cong_kk^{[n]}
\]
\end{proof}

\begin{corollary}\label{Asanuma} For $m\ge 2$, $A_m\not\in\mathcal{U}_k(3,{\bf G})$.
\end{corollary}

In fact, Gupta found counterexamples to cancellation in positive characteristic for every dimension $n\ge 3$; see \cite{Gupta.14b}. 
Thus, for each such counterexample $R$, we have $R\not\in\mathcal{U}_k(n,{\bf G})$. 


\section{Conclusion}\label{conjecture}

We conclude with some remarks and a conjecture.

\begin{remark}\label{Klein} {\rm The ring $B=k[x,y,z]/(x^5+y^3+z^2)$ is also rigid (see \cite{Freudenburg.17}, Thm.\,9.7), but it is not known whether it satisfies the stronger property described in 
{\it Theorem\,\ref{stable}}.}
\end{remark}

\begin{remark} {\rm If $k$ is not algebraically closed, then in general, $\mathcal{B}(k)\ne \mathcal{U}_k(2,{\bf G},{\bf R})$. 
For example, {\it Theorem\,\ref{Samuel}} shows that 
\[
\R [x,y,z]/(x^2+y^2+z^{2k+1})
\]
is a rational UFD which admits a positive $\Z$-grading for all integers $k\ge 0$.}
\end{remark}

\begin{remark} {\rm The set $\mathcal{U}_k(2,{\bf A},{\bf G},{\bf P},{\bf R})$ contains more than just $k[x,y]$. 
For example, it is well-known that the ring
\[
B=k[x,y,z]/(x^5+y^3+z^2)
\]
is the ring of invariants for an action of an icosahedral group on the plane, so $B\subset k^{[2]}$. For a specific polynomial parametrization, see \cite{Riemenschneider.77}, {\S}2.E. 
Likewise, Russell \cite{Russell.81} gives the subalgebra:
\[
B'=k[uT^b, vT^c,T]\subset k[u,v] \quad {\rm where} \quad  T=u^c+v^b\,\, ,\,\, \gcd(b,c)=1
\]
Then $B'\cong k[x,y,z]/(x^{bc+1}+y^b+z^c)$.
}
\end{remark}

\begin{remark}\label{polarized} {\rm For $R\in\mathcal{U}_k(n,{\bf A},{\bf G})$, let $\mathfrak{g}$ be a positive $\Z$-grading of $R$ given by $R=\bigoplus_{i\in\N}R_i$. 
Given $a\in\N$, define the homogeneous subalgebra:
\[
R^{(a)}=\bigoplus_{i\in\N}R_{ia}
\]
Note that $R^{(1)}=R$ and that $R^{(b)}\subset R^{(a)}$ when $a$ divides $b$. In \cite{Mori.77}, Mori shows that there exists a unique positive integer $m$ with the property:
\begin{center}
$R^{(a)}$ is a UFD if and only if $a$ divides $m$
\end{center}
This integer $m$ is called the {\bf index} of the pair $(R,\mathfrak{g})$.
}
\end{remark}

Finally, we propose the following characterization of the affine space $\A^n_k$.

\begin{quote} {\bf Conjecture.} Let $k$ be an algebraically closed field and $n\in\Z$ positive.\\
If $B\in\mathcal{U}_k(n,{\bf A},{\bf G},{\bf R})$ and $X={\rm Spec}(B)$ is smooth, then $X\cong_k\A^n_k$.
\end{quote}

The conjecture is true if $\dim_kX\le 2$: The case $\dim_kX=1$ follows by {\it Corollary\,\ref{dim-1}}, and the case $\dim_kX=2$ follows by {\it Theorem\,\ref{classify}}. 

\medskip

\noindent {\bf Acknowledgment.} The authors wish to acknowledge Daniel Daigle of the University of Ottawa for helpful comments regarding this paper, and for pointing out Eakin's lemma to us. 
The second author wishes to express his gratitude to members of the Department of Mathematics at Western Michigan University, which he visited during the fall of 2018. Much of the research for this paper was conducted during that time. 


\bigskip

\noindent \address{Department of Mathematics\\
Western Michigan University\\
Kalamazoo, Michigan 49008}\\
USA\\
\email{gene.freudenburg@wmich.edu}
\bigskip

\noindent\address{Graduate School of Science and Technology\\
Niigata University\\
8050 Ikarashininocho, Niigata 950-2181\\
Japan\\
\email{t.nagamine14@m.sc.niigata-u.ac.jp}

\end{document}